\documentclass[11pt]{amsart}
\pdfoutput=1
\usepackage{amsmath, amsthm, amssymb}
\usepackage{fullpage}

\usepackage{txfonts}

\usepackage[usenames]{color}

\usepackage{hyperref}
\usepackage{url}
 
\usepackage{enumerate}

\numberwithin{equation}{section}

\newtheorem{theorem}{Theorem}[section]
\newtheorem{lemma}[theorem]{Lemma}

\theoremstyle{definition}
\newtheorem{definition}[theorem]{Definition} 

\theoremstyle{remark}
\newtheorem *{remark}{Remark}
\newtheorem *{quest}{Question}
\newtheorem *{warn}{WARNING}

\newcommand\C{{\mathbb C}} 
\newcommand\R{{\mathbb R}} 
\newcommand\Z{{\mathbb Z}} 

\newcommand{\res}{{\bigg|}} 
\renewcommand{\l}{\left(}
\renewcommand{\r}{\right)}

\newcommand{\hrad}{r_{I}}
\newcommand{\cdom}{C^{k-2,\alpha}_{\gamma-2} }
\newcommand{\dom}{C^{k,\alpha}_\gamma}
\newcommand{\wlp}[3]{C^{#1,#2}_{ #3}}

\newcommand{\LL}[1]{L_{#1}}

\begin{document}

\title{The Nirenberg Problem for Conical Singularities}

\author{Lisandra Hernandez-Vazquez}
\address{Stony Brook University}
\email{lisandra@math.stonybrook.edu}
\date{June 1, 2018}


\begin{abstract}
We propose a new approach to the question of prescribing Gaussian curvature on the 2-sphere with at least three conical singularities and angles less than $2\pi$, the main result being sufficient conditions for a positive function of class at least $C^2$ to be the Gaussian curvature  of such a conformal conical metric on the round sphere. Our methods particularly differ from the variational approach in that they don't rely on the Moser-Trudinger inequality. Along the way, we also prove a general precompactness theorem for compact Riemann surfaces with at least three conical singularities and angles less than 2$\pi$.  
\end{abstract}

\maketitle 

\section{Introduction}
Conical surfaces have been extensively studied by many in different contexts (see \cite{song, ramos, mazzeoricci, chen} for some examples). This class of objects consists of Riemann surfaces $\Sigma$, without boundary, equipped with a Riemannian metric $g$, points $p_1, \dots p_n$ and neighborhoods $U_i$ of each $p_i$ such that 
\begin{enumerate}
\item $g$ is of class $C^k$ on  $\Sigma- \{p_1, \dots, p_n\}$
\item there exist local holomorphic coordinates $z_i: U_i \to \C$ on each $U_i$ and bounded functions $\phi_i(z_i)$ of class $C^k$ such that 
\begin{equation}\label{isothermalcharts}
g = e^{2\phi_i(z_i)}|z_i - z_i(p_i)|^{2\beta_i}|dz_i|^2
\end{equation}
\end{enumerate}
where $\beta_i >-1$ for all $i$. The points $p_1, \dots p_n$ are called \textit{cone points} with cone angles $\theta_i := 2\pi (\beta_i+1)$. If a metric $g$ on $\Sigma$ satisfies the conditions above, we say $g$ is a conical metric representing the divisor  $\beta = \sum_{i=1}^n \beta_ip_i$, and write $(\Sigma, g, \beta)$ for brevity. Many of the usual topological invariants defined for smooth surfaces extend to conical ones. For instance, one defines the \textit{generalized Euler characteristic} for the conical surface $(\Sigma, g, \beta)$ by
$$\chi(\Sigma, \beta) := \chi(\Sigma) + \sum_{i=1}^n \beta_i$$
and can further prove a corresponding Gauss-Bonnet theorem (\cite{troyanov}), which states that if $K_g$ is the Gaussian curvature of a conformal conical metric $g$  (see Section 2 for the definition) and $dA_g$ is its area element, then 
\begin{equation}\label{gauss bonnet}
\int_{\Sigma} K_g dA_g = 2\pi\chi(\Sigma) + 2\pi \sum{\beta_i} = 2\pi\chi(\Sigma, \beta)
\end{equation}
A classical problem for conical surfaces is characterizing those smooth functions which arise as Gaussian curvatures of a pointwise conformal metric. Equivalently, one asks for necessary and sufficient conditions for existence of solutions to the Gaussian curvature equation of a pointwise conformal metric $\tilde{g} = e^{2u}g$,  with $g$  a conical metric representing a given divisor $\beta$. Such equation is given by 
\begin{equation}\label{gaussianeq}
K= e^{-2u}(K_g - \Delta_g u)
\end{equation}
where $K_g$ is the curvature of the metric $g$. In \cite{troyanov}, Troyanov proves several existence and uniqueness results analogous to the ones shown by Kazdan-Warner in \cite{kazdanwarner1} for smooth Riemann surfaces. Since the pioneering work of Troyanov in \cite{troyanov}, where he uses a variational method to find solutions to (\ref{gaussianeq}), several other methods have been employed. These include complex analytic ideas \cite{eremenko}, Ricci flow \cite{mazzeoricci}, minmax theory \cite{carlotto} and recently, synthetic geometric methods when the surface is a sphere \cite{mondello}. In the case of constant curvature, there is a complete existence theory developed over the years \cite{eremenko, mcowen, luotian, troyanov} for conical surfaces with at least three conical singularities and angles less than $2\pi$. In particular, it has been observed by many that a necessary condition for the existence and uniqueness of such conformal conical metric of constant curvature $1$ on $S^2$ is 
\begin{equation} \label{troyanovcondition}
\sum_{i\neq j} \beta_i < \beta _j,  \textit{  for all  } j
\end{equation}

Condition (\ref{troyanovcondition}) is known as the \textit{Troyanov condition} and we refer the reader to \cite{mazzeoweiss} for a geometric interpretation. A particularly delicate case is when the conical surface is a sphere, referred to in the literature as the \textit{singular or conical} Nirenberg problem. The classical  Nirenberg problem asks to characterize those smooth functions on the sphere which arise as the Gaussian curvature of a metric that is pointwise conformal to the round metric of curvature $1$. Even though much progress has been done over the years (see for instance, \cite{changyang, han, ji, ding, changliugursky, changliu}), the problem still remains open in full generality. In passing to the singular case, one now allows metrics in a given conformal class to have conical singularities. While there are some results on this topic in the literature \cite{marchis, troyanov}, many questions remain unanswered. In this work we propose a new approach to the singular Nirenberg problem when there are at least three cone points and the angles are less than $2\pi$. Using our methods we find sufficient conditions for a function $K$ to arise as the Gaussian curvature of a conformal conical metric in this setting. Specifically, we show
\begin{theorem}\label{main}
Suppose $n \geq 3$, and $\beta = \sum_{i=1}^n \beta_i p_i$ is a divisor on $S^2$ satisfying the Troyanov condition (\ref{troyanovcondition}) and there exists $i,j,k$ distinct for which $\beta_i , \beta_j, \beta_k$ are all distinct. Assume $\chi(S^2, \beta)>0$ and let $g_\beta$ be the unique conical metric on $S^2$ representing the divisor $\beta$ of Gaussian curvature $K_\beta = 1$. Then a function $K$ on $S^2$ is the Gaussian curvature of a metric $g$ conformal to $g_\beta$ if $K$ is a positive function in $ \cdom$, $k\geq 2, \alpha\in (0,1)$, where $\gamma = (\gamma_1, \dots, \gamma_n) \in \R^n, \gamma_i \neq \frac{m}{\beta_j}$, $\gamma_i >0$ and $m$ is an integer.  
\end{theorem}

The space $\cdom$ consists of H\"{o}lder continuous functions which are $(k-2)$-times differentiable and satisfy growth conditions near the cone points that are determined by the weights $\gamma$ (for a precise definition see Section 2). Following the ideas in recent work of Anderson ~\cite{andersonnirenberg}, our approach is to choose appropriate Banach spaces such that the differential operator defined by (\ref{gaussianeq}) is a proper Fredholm map of index zero. For such operators, there is a well-defined notion of degree, and one can perform a degree count in order to study the surjectivity of this map. In establishing properness, a preliminary step is a compactness theorem for conical surfaces.  Although compactness theorems for conical surfaces have been shown in \cite{debin} for the uniform topology and in ~\cite{ramos}  for the $C^{m,\alpha}$ topology but angles less than $\pi$, to our knowledge no result of the form required for our proof of  Thm \ref{main} exists in the literature. Thus in Section 2, we show

\begin{theorem}\label{compactthm}:
Let $\Sigma$ be a compact Riemann surface without boundary and fix a divisor $\beta= \sum_{j=1}^n \beta_j p_j$ on $\Sigma$ such that  $-1 < \beta_j <0$ for all $j$. If $M_i= (\Sigma, g_i, \beta$) is a sequence of smooth conformal conical metrics on $\Sigma$ representing the divisor $\beta$ such that there exist constants $D_0, \Lambda, v_0>0$ for which 

\begin{enumerate}
\item $diam(M_i) \leq D_0$
\item  there exists $t_0>0$ such that $vol_{g_i}(B(r)) \geq v_0$ for every $r \leq t_0$ 
\item $||K_{g_i}(x)||_{0} \leq \Lambda$ away from the cone points
\end{enumerate}
Then for any $\gamma\in \R^n$, there exists a subsequence of $(g_i)$, $\wlp{2}{\alpha}{}$ diffeomorphisms $F_i:\Sigma \to \Sigma$ and a $\wlp{1}{\alpha}{\gamma}$ conformal conical metric $g$ representing a divisor $\beta'$ such that 
\begin{equation}
||(F^*_i g_i)_{st}- g_{st}||_{1, \alpha; \gamma} \to 0
\end{equation}
as $i\to \infty$, where  $\beta' = \sum_{j=1}^m \beta_j' q_j$ with $-1<\beta'_j \leq 0$, $m\leq n$.
\end{theorem}

This result, however, only guarantees control of the metrics $g=e^{2u_i}g_\beta$ in a weighted $C^{m,\alpha}$ topology modulo diffeomorphisms. In order to obtain properness of the differential operator defined by (1.3), one needs to ensure the conformal factors $u_i$ themselves converge on a subsequence. As in \cite{andersonnirenberg}, one can further show that also in the conical case we actually have control of the metrics modulo the conformal group of $g_\beta$, where Astala's theorem (\cite{astala}) plays a central role in this step. To obtain that the conformal factors $u_i$ themselves converge,  we rely on the fact that a conformal conical metric on a sphere with at least three cone points has compact conformal group (see Thm \ref{finiteconformal}). 

\subsection*{Acknowledgements} I would first like to thank my advisor Michael Anderson for all his help and advice in the completion of this project. I would also like to thank Dror Varolin and Frederik Benirschke for their time in discussing aspects of this work with me.

 \section{Compactness Theorems For Conical Surfaces }

As mentioned in the introduction, a $C^k$ ($k\geq 2$) conformal conical metric has an associated curvature function defined on the complement of the cone points. In fact, using (\ref{isothermalcharts}), we can explicitly compute the Gaussian curvature in a neighborhood of any given point as stated in the following lemma.
\begin{lemma}\label{gaussiancurvature}
Let $\Sigma$ be a compact Riemann surface and $g$ a conformal conical metric representing the divisor $\beta$. If $z$ is a holomorphic coordinate in a neighborhood of $p \in \Sigma$ such that $z(p) = 0$ and the conformal conical metric is given locally by
$g = e^{2u}|z|^{2\beta}|dz|^2$, then the Gaussian curvature $K$ of g satisfies 
\begin{equation}\label{gaussiancone}
K = -\dfrac{e^{-2u}\Delta u}{|z|^{2\beta}}
\end{equation}
for $z\neq 0$, where $\Delta$ is the Euclidean laplacian. 
\end{lemma}

A natural question is whether there is an associated compactness or precompactness theorem for the class of conical Riemann surfaces with curvature, diameter and volume bounds. We address this question following the ideas in the classical work of Anderson (\cite{1anderson1}, see also \cite{AnCh}). To begin, a chart $H$ will be referred to as an \textit{isothermal chart} if there exists a holomorphic coordinate $z$ such that 

 \begin{equation}
  (H^{-1})^*g = e^{2\phi(z)}\prod_{i=1}^m |z- z^j|^{2\alpha_j}|dz|^2
 \end{equation}
 Observe that the existence of such charts is guaranteed by the uniformization theorem. 
 
 In order to directly apply our results in Section 3, we will work with weighted H\"{o}lder spaces, which we define here as follows. To set things up, we first introduce a slightly different but equivalent definition of a conical surface as the one given in the introduction. Let $\Sigma$ be a compact Riemann surface without boundary and fix a $C^k$,  $k\geq 2$, metric $g_0$ on $\Sigma$, points $p_1, \dots p_n$. Define $\Sigma^* = \Sigma - \{p_1, \dots p_n\}$. For $R>0$ small enough, let $B_R(p)$ denote a geodesic ball in $(\Sigma, g_0)$ of radius $R$ centered at $p$. Set $ \Sigma_R := \Sigma - \cup_{i=1}^n \bar{B}_R(p_i)$, where $\bar{B}_R(p_i)$ denotes the closure of $\bar{B}_R(p_i)$ in $(\Sigma, g_0)$. 

\begin{definition}\label{radiusfunction}
A smooth function $\rho:  \Sigma^* \to (0, 1]$ is a \textit{radius function} on $( \Sigma^*, g_0)$, if for $R>0$ small enough, $\rho \equiv 1$ on $ \Sigma_R$ and for each $i=1, \dots n$, there exist holomorphic coordinates $z_i$ on $B_R(p_i)$ such that $\rho(z_i) = O(|z_i|)$ on $B_R(p_i)$.

We further define $\rho^\gamma$ for $\gamma = (\gamma_1, \dots, \gamma_n)$ in $\R^n$ as follows 
\begin{align*}
\rho^\gamma = \rho^{\gamma_i} & \text{ on } B_R(p_i)\\
\rho^\gamma \equiv 1 & \text{ otherwise }
\end{align*}
Moreover, for $a \in \R$, set $\gamma + a := (\gamma_1 +a, \dots \gamma_n +a)$. We say $\gamma \leq \delta$, $\gamma, \delta \in \R^n$ provided that $\gamma_i \leq \delta_i$ for each $i = 1, \dots n$. 
\end{definition}

\begin{definition}\label{conicalconformaldef}
We say $g_\beta$ is a \textit{conical metric} on $(\Sigma, g_0)$ of class $C^k$ ($k\geq 2$) representing the divisor $\beta = \sum_{i=1}^n \beta_i p_i$ if $g_0$ is of class $C^k$ and there exists a smooth radius function $\rho: \Sigma^*\to (0,1]$ such that
\begin{equation}
g_\beta = \rho^{2\beta} g_0
\end{equation}

Moreover, $g$ is called a \textit{conformal conical metric} on $(\Sigma, g_0)$ of class $C^k$ representing the divisor $\beta$ if there exists a smooth positive function $u: \Sigma \to \R$ such that 
\begin{equation}
g = e^{2u}g_\beta
\end{equation}
where $g_\beta$ is a conical metric on $\Sigma$ of class $C^k$ representing the divisor $\beta$. 
\end{definition}
For $g_\beta$ a conical metric on $(\Sigma, g_0)$ representing a divisor $\beta$ and $\rho$ a radius function on $(\Sigma^*, g_0)$ such that $g_\beta = \rho^{2\beta}g_0$,  if $u \in C^k_{loc} (\Sigma^*)$ and $\gamma \in \R^n$, set 
\begin{equation}
||u||_{C^k_\gamma} := \sum_{j=0}^k \sup_{x \in \Sigma^*} |\rho(x)^{-\gamma+j} \nabla^j u(x)|
\end{equation}
Define the space of $\wlp{k}{\alpha}{\gamma}(\Sigma, \beta)$ functions on $(\Sigma, \beta)$ to be 
\begin{equation}
\wlp{k}{\alpha}{\gamma}(\Sigma, \beta) = \bigg\{ u\in C^{k}_{loc}(\Sigma^*)  : ||u||_{k, \alpha; \gamma}< \infty \bigg\}
\end{equation}
 where the norm $||\cdot||_{k, \alpha; \gamma}$ is given by 
\begin{equation}
||u||_{k, \alpha; \gamma} := ||u||_{C^{k, \alpha}_{\gamma}} = ||u||_{C^k_\gamma} + [\nabla^k u]_{\alpha, \gamma-k}
\end{equation}
and
\begin{equation}
[\nabla^k u]_{\alpha; \gamma} = \sup_{x\neq y, d(x,y) < inj(x)} \min(\rho(x)^{-\gamma}, \rho(y)^{-\gamma} )\dfrac{|\nabla^ku(x) - \nabla^ku(y)|}{d(x,y)^\alpha}
\end{equation}

It is well known that the normed spaces $(\wlp{k}{\alpha}{\gamma}, ||\cdot||_{k,\alpha; \gamma})$ are Banach for any $\gamma$ (see for instance \cite{pacard}). 
One can further define the weighted Sobolev spaces $W^{k,p}_\gamma$, as in \cite{behr}. Next, we introduce the notion of the \textit{isothermal radius}, which plays the same role as that of the harmonic radius in \cite{1anderson1, AnCh}.
\begin{definition}\label{defharmradius}\textit{The Isothermal Radius:}
Let $(\Sigma, g, \beta)$ be a complete Riemann surface without boundary with $g$ a conformal conical metric on $\Sigma$ representing the divisor $\beta = \sum_{i=1}^n \beta_ip_i$.  Let $x\in \Sigma$. Given a constant $C>1$, $\alpha\in (0,1), \gamma\in \R^n$, we define the isothermal radius  $\hrad =  r_{I}(g, x, C, \alpha, \gamma) $ as the largest number such that on the geodesic ball $B(x, r_{I}(x))$ there exists an isothermal coordinate chart $H: B(x, r_{I}(x)) \to B_0(R) \subset \C$ with 
 
 \begin{equation}\label{isothermal}
  (H^{-1})^*g = e^{2\phi(z)}\prod_{i=1}^m |z- z^j|^{2\alpha_j}|dz|^2
 \end{equation}
where $z$ is a holomorphic coordinate on $B_0(R)$, $z^j = H(p_j)$ correspond to  cone points, $1 \leq m \leq n$ , $-1< \alpha_j$ and $\phi(z): B_0(R) \to \R$ is smooth and satisfies
 \begin{enumerate}  
\item[A1.] $\dfrac{1}{C}\leq \phi(z) \leq C$ 

\item[B1.] $\sum_{0\leq |\mu| \leq 1} \hrad^{-\gamma +|\mu|} \sup_x|\partial^\mu \phi(x)|+ \sum_{|\mu|=1} \hrad^{-\gamma +\alpha +1 } \sup_{x\neq y}\frac{|\partial^\mu \phi(x) - \partial^\mu \phi(y)|}{d(x,y)^\alpha} \leq C -1 $
\end{enumerate}

We also define the isothermal radius of $(\Sigma, \beta, g)$ by 
\begin{equation}
\hrad(\Sigma) = \inf_{x\in \Sigma } \hrad(x)
\end{equation}
\end{definition}

Observe that $\alpha_j$ either coincide with the $\beta_j$ or are zero in neighborhoods with no conical points. In the following lemma we prove properties of the isothermal radius that will be needed for the upcoming blow-up argument in the proof of Theorem \ref{compactthm}. All of these facts are true for the harmonic radius and the proofs presented here are essentially the same as in the works of \cite{1anderson1, AnCh, hebey}, only slightly modified to fit our setting. 

\begin{lemma} \label{propertiesisorad}
Let $(\Sigma, \beta, g)$ be a Riemann surface with a conformal conical metric $g$ representing the divisor $\beta$ and let  $\hrad: \Sigma \to \R$ be the  isothermal radius. Then the following hold 
\begin{enumerate}
\item $\hrad(x)$ is positive and pointwise continuous on $\Sigma$
\item If $F: (\Sigma,g) \to (\Sigma', g')$ is an isometry, then $$F^*\hrad{(\Sigma')} = \hrad{(\Sigma)}$$
\item $\hrad$ scales as the distance: $\hrad(\lambda^2 g, x) = \lambda \hrad(g, x)$
\end{enumerate}
\end{lemma}

\begin{proof}
Proof of (1). 
For the positivity, observe that for any $x \in \Sigma$ we can find a $0 \geq\mu > -1$, a neighborhood $U $ of $x$  and a holomorphic coordinate $z: U \to \C$ such that 
$g = e^{2\phi} |z|^{2\mu} |dz|^2$ for some smooth $\phi(z): U \to \R$.
Clearly, if $x$ is any of the cone points, such coordinates exist by definition once we choose $\mu = \frac{\theta}{2\pi}$ where $\theta$ is the cone angle. If on the other hand $x$ is a smooth point, then such coordinates are guaranteed by the Uniformization Theorem: we can find a neighborhood $U$ of $x$ and a holomorphic coordinate $z: U \to \C$  such that
$ g = e^{2\phi} |dz|^2$. Finally observe that conditions $A1, B1$ always hold on a fixed conical surface in a given isothermal chart.

For the continuity, given any $x$ close enough to $y$, we can find a ball of radius $a$ centered at $x$ which contains all the cone points in $B(y, \hrad(y))$ and is contained in $B(y, \hrad(y))$. Therefore if $z$ is a holomorphic coordinate in $B(y, \hrad(y))$ such that 
$$  g = e^{2\phi(z)}\prod_{i=1}^m |z- z_i|^{2\alpha_i}|dz|^2$$
with $\phi(z)$ satisfying conditions $A1,B1$ of Definition \ref{defharmradius}, by restricting the coordinates $z$ to $B(x, a)$, we get a holomorphic coordinate on this ball such that in this coordinate the same function $\phi(z)$ still satisfies conditions $A1,B1$ for the same $C,\alpha, \gamma$. Therefore $a \leq \hrad(x)$. Using this one can directly show 
$$|\hrad(y)-\hrad(x)| = |a + d(x,y) - \hrad(x)|\leq \epsilon$$
as wanted. 

Proof of (2). We will actually prove something stronger than the statement of (2): computing the isothermal radius with respect to $g'$ at a point $F(x)$ gives the same result as computing the isothermal radius with respect to $g$ at the point $x$. To begin, fix $x\in \Sigma$ and let $H':B(F(x), \hrad{F(x)}) \to B_0(R) $ be an isothermal coordinate chart where 
\begin{equation}
(H'^{-1})^*g' = e^{2\phi(z)}\prod_{i=1}^m |z- z_i|^{2\alpha_i}|dz|^2
\end{equation}
and $\phi(z)$ satisfies conditions $A1, B1$ of Definition \ref{defharmradius}. The set $U = F^{-1} (B(F(x), \hrad{F(x)}))$ is open in $\Sigma$ and contains $x$. If $B(x, R)$ is the largest geodesic ball centered at $x$ which is still contained in $U$, then $H:= H' \circ F: B(x, R) \to B_0(R)$  is an isothermal coordinate chart for the metric $g$ on $B(x, R)$.  Moreover, the metric $g$ on the ball $B(x, R)$ has the form (\ref{isothermal}) above, where the conformal factor can be computed to be $2\phi(F(z)) + ||DF||^2$ and satisfies the bounds $A1, B1$ in these coordinates since $||DF||_{(F^{-1})^*g}=1$. We have thus found a ball centered at $x$ in which all the conditions of the definition of the isothermal radius are satisfied, so we conclude that its radius $R \leq \hrad(x)$. In particular, $\hrad  (F(x),g' )\leq \hrad (x, g)$. To obtain the opposite inequality, we just follow the same argument with $F^{-1}$ in place of $F$.\\
Proof of (3). We want to show that $\hrad(\lambda^2g, x) = \lambda \hrad(g, x)$ for any nonzero $\lambda$. Start with a conical metric $g$ and after having chosen a holomorphic coordinate for which we can write the metric $g$ as in (\ref{isothermal}) of Definition \ref{defharmradius}, with $\phi(z)$ satisfying the bounds $A1, B1$, it is straightforward to check that the metric $\lambda^2g$ satisfies the same bounds in the coordinates $w = \lambda z$. 
\end{proof}

We now turn to the question of convergence.
For $k \geq 2, \alpha \in (0,1), \gamma \in \R^n$, we say a conformal conical metric $g$ on $(\Sigma, g_0)$ representing the divisor $\beta$ is of class $\wlp{k}{\alpha}{\gamma}$ if in an isothermal chart the coefficients $g_{ij}$ of $g$ are bounded in $\wlp{k}{\alpha}{\gamma}$. Moreover, a sequence of conformal conical metrics $(\Sigma, g_i, \beta)$ of class $\wlp{k}{\alpha}{\gamma}$ converges in $\wlp{k}{\alpha}{\gamma}$ to a  surface $(\Sigma', g)$ provided that there exists a sequence of  $\wlp{k+1}{\alpha}{}$ diffeomorphisms $F_i: \Sigma' \to \Sigma$ such that for all $i$ large enough 
\begin{equation}\label{weightedconvergence}
||(F^*_ig_i) - g||_{k,\alpha; \gamma} \to 0
\end{equation}
in any chart of a $C^\infty$ subatlas of the complete $C^\infty$ atlas of $\Sigma$.  The following lemma addresses the continuity of the isothermal radius in the $C^{1,\alpha}_{\gamma}$ topology. 

\begin{lemma}\label{uppersemicont}
For $\alpha \in (0,1)$, $\gamma \in \R^n$, the  isothermal radius is continuous under $\wlp{1}{\alpha}{\gamma}$ convergence of a sequence of conical metrics $(\Sigma, g_i, \beta)$ representing the divisor $\beta$.
\end{lemma}

\begin{proof}
Let $(\Sigma, g_i, \beta)$ be a sequence of conformal conical metrics on $\Sigma$ representing the divisor $\beta$ and $x \in \Sigma$. As before, $\Sigma^* = \Sigma - \{p_1, \dots p_n\}$, where $p_1, \dots p_n$ are the cone points. Fix $\alpha \in (0,1), \gamma \in \R^n$. Assume that the sequence $(g_i)$ converges in $\wlp{1}{\alpha}{\gamma}$ to a $\wlp{1}{\alpha}{\gamma}$ metric $g$ on $\Sigma$. By the continuity of the isothermal radius we mean explicitly that the following two inequalities hold: given $C>1$
\begin{equation}\label{in1}
\hrad(g, x, C) \geq \limsup_{i\to \infty} \hrad(g_i, x, C)
\end{equation}
and for any $0< \epsilon < C-1$
\begin{equation}\label{in2}
\hrad(g, x, C-\epsilon) \leq \liminf_{i \to \infty} \hrad(g_i, x, C)
\end{equation}

For simplicity, set $r_i := \hrad(g_i, x, C)$ and let $H_i: B(x,r_i) \to \C$, $H_i(y) = (H_i^1(y), H_i^2(y))= z_i \in \C$ be isothermal coordinate charts in which the metrics $g_i$ satisfy $A1, B1$ of Definition \ref{defharmradius}. 

We begin with some preliminary claims, the first being that for any $r \leq \limsup r_i$, a subsequence of the isothermal charts $H_i$ converges in $C^{2,\alpha}$ to an isothermal chart $H: B(x, r) \to \C$, where $B(x,r)$ is a geodesic ball for the metric $g$. To this end, suppose $ (x_1(y), x_2(y)) \in \C$ is any given local coordinate chart on $B(x, r)$. Observe that it is implicit in the fact that the charts $H_i$ are isothermal that the coordinate functions $H_i^k$, for $k=1,2$ are harmonic. In other words, 

\begin{align}\label{eq1}
(g_i)^{st} \frac{\partial^2 H_i^k}{\partial x_s \partial x_t}&= (g_i)^{st} (\Gamma_i)^l_{st}\frac{\partial H^k_i}{\partial x_l}
\end{align}
where $(g_i)^{st}$ are the components of $g_i$ in the coordinates $(x_1, x_2)$ and $(\Gamma_i)^l_{st}$ are the Christoffel symbols for $g_i$ in these coordinates. Now, condition $A1$ of Definition (\ref{defharmradius}) implies that in the charts $H_i$, the components of the metrics $g_i$ satisfy 
\begin{align}\label{A1bounds}
\frac{1}{C} \prod_{j=1}^{m_i} |z_i - z_i^j|^{\alpha_j} \delta_{kl}\leq (g_i)_{kl} \leq C \prod_{j=1}^{m_i} |z_i - z_i^j|^{\alpha_j} \delta_{kl}
\end{align} 
where we write $z_i = H_i^1+ i H_i^2$ and the inequality holds as bilinear forms. It then follows from (\ref{A1bounds}) and the fact that the metrics converge in $C^{1,\alpha}_{\gamma}$ that the charts $H_i$ are bounded in $C^1$ on $\Sigma^*$. Using standard elliptic estimates for  (\ref{eq1})  (see \cite{gt}, for instance), we obtain that for each $k=1, 2$, the sequence $(H_i^k)$ is bounded in $\wlp{2}{\alpha}{}(\Sigma^*)$. Therefore, by the Arzela-Ascoli theorem, we have that for each $k=1, 2$, the sequences $(H^k_i)$ converge weakly on a subsequence in $C^{2,\alpha'}$ on $\Sigma^*$ for $\alpha' \leq \alpha$. In fact, repeating this argument for  $H_q^k - H_n^k$  in place of $H_i^k$ one can see that for each $k = 1,2$, the sequences $(H_i^k)$ are in fact Cauchy. Therefore, they converge strongly on a subsequence in $C_{loc}^{2,\alpha}$ to a limiting map $H: B(x,r)- \{p_1, \dots p_m\} \to \C$, where $H_i(p_j) = z^j$ for all $i$. Since the property of being an isothermal chart is preserved under $\wlp{2}{\alpha}{}$ convergence, $H$ is an isothermal chart for the metric $g$. 
Moreover, since there are finitely many cone points, by passing to a subsequence, we may assume $m_i$ in (\ref{A1bounds}) is independent of $i$ and $z_i^j= z^j$ for all $i$. Therefore, $g$ can be written in the chart given by $H$ as

\begin{equation}
(H^{-1})^*g = e^{2\phi(z)} \prod_{j=1}^{m} |z- z^j|^{\alpha_j}|dz|^2
\end{equation}
where $\phi(z)$ is a smooth function that satisfies conditions $A1, B1$ of Definition \ref{defharmradius}. Hence for any $r \leq \limsup r_i$, a subsequence of the isothermal charts $H_i$ converges in $\wlp{2}{\alpha}{}$ to an isothermal chart $H: B(x, r) \to \C$ for the metric $g$. Observe that this argument also shows that if a sequence of conformal conical metrics representing a fixed divisor $\beta$ converges in $\wlp{1}{\alpha}{\gamma}$ to a metric $g$, then $g$ has at most as many cone points as the sequence $g_i$ and no other types of singularities. 

Our second preliminary claim is that if $(\Sigma, g, \beta)$ is any complete Riemann surface without boundary with a conformal conical metric $g$ representing the divisor $\beta$, $x\in \Sigma$, $\gamma \in \R^n$,  then for any $1 \leq C'  \leq C < \infty$,
\begin{equation}\label{3.9}
\hrad(C')(x) \leq \hrad(C)(x)
\end{equation}
and for any $C>1$

\begin{equation}\label{3.10}
\lim_{\epsilon\to 0^+} \hrad(C+ \epsilon)(x) = \hrad(C)(x)
\end{equation}

The first inequality follows from the definition, hence to prove the claim, it's enough to show that for any $C>1$
\begin{equation}
\limsup_{\epsilon\to 0^+} \hrad(C+\epsilon, x) \leq \hrad(C, x)
\end{equation}

Fix $r < \limsup \hrad(C+ \epsilon, x)$. For a decreasing sequence of $\epsilon >0$ converging to $0$, there are isothermal coordinate charts $H_\epsilon$ on $B_x(r)$ satisfying conditions $A1, B1$ of Definition \ref{defharmradius} with $C+\epsilon$ in place of $C$ and $r$ in place of $\hrad$. By the same arguments as above, we get that a subsequence of $H_\epsilon$ converges in $C_{loc}^{2,\alpha}$ to a limiting chart $H$. As before, the bounds $A1, B1$ are preserved under $\wlp{2}{\alpha}{}$ convergence, hence $\hrad(C, x) \geq r$. Since $r < \limsup \hrad( C +\epsilon, x)$ was arbitrary, this proves the claim.

We're now ready to prove the first inequality (\ref{in1}). As before, let $r_i = \hrad (g_i)$. We may suppose $\limsup r_i >0$. The arguments above show that convergence of the metrics in $\wlp{1}{\alpha}{\gamma}$ implies convergence of the isothermal charts $H_i$ in $C_{loc}^{2,\alpha}$. Once again, the bounds $A1, B1$ are preserved, so that $\hrad(g, C) \geq r$ for any $r \leq \limsup r_i$. Therefore we get the first inequality $\hrad(g, C) \geq \limsup \hrad(g_i, C)$. 

Now fix $r < \hrad(g, C)$. To obtain the second inequality (\ref{in2}), let $H: B(x,r) \to B_0$ be an isothermal coordinate chart for $g$, so that $(H^{-1})^*g = e^{2\phi(z)} \prod |z-z^j|^{2\alpha_j}|dz|^2$, with $z$ a holomorphic coordinate on $B_0 \subset \C$. Let $\Delta_i$ be the Laplace operator for the metric $g_i$. In the coordinate $z$, the Laplacian for the metrics $g_i$ has the form 
\begin{equation}
\Delta_i = e^{-2 (\phi_i(z)+ \sum_{j=1}^{m} \alpha_j\log |z- z^j|)}\Delta
\end{equation}
where $\Delta$ is the Euclidean laplacian. As observed before, we may assume the cone points $z^j$ are the same for each $i$ after passing to a subsequence. 
Now, if $w_i$ are solutions to 
\begin{align}
\Delta_i w_i &= 0 \text{ on }  B\\
w_i &= z \text{ on } \partial B
\end{align}
then the functions $u_i:= z- w_i$ are harmonic and vanish on the boundary of $B$. Since, for each $i$, the function $z$ also solves the boundary value problem (2.22-2.23) it follows from uniqueness that in fact $u_i = 0$. Therefore,  $\lim_{i\to \infty} ||u_i||_{2,\alpha} = 0$ and we have that for any compact subset $B' \subset B$ and for any $i$, $H$ is an isothermal coordinate chart for $g_i$.
Now, since the metrics $g_i$ converge to $g$ in $\wlp{1}{\alpha}{\gamma}$, $H$ is an \textit{isothermal} coordinate chart for $g_i$ in which the bounds of Definition \ref{defharmradius} are satisfied with constants $C_i \to C$ as $i\to \infty$. Using (\ref{3.9}), (\ref{3.10}), we have that for any $\epsilon >0$
$$r \leq \liminf r_I( g_i, C_i) \leq \liminf r_I(g_i, C+\epsilon)$$

Since $r \leq r_I(g, C)$ was arbitrary, this ends the proof of the second inequality. 
\end{proof}

An important property of the harmonic radius in the smooth case is that Euclidean space has infinite harmonic radius. The isothermal radius satisfies an analogous condition, but in this case the model space is a flat Riemann surface $M$ with finitely many cone points and angles less than $2\pi$, which is noncompact, complete and of \textit{quadratic area growth }, in the sense that for any $x \in M$ and any $r>0$
\begin{equation}\label{quadratic}
\frac{1}{V} r^2 \leq vol(B(r, x)) \leq V r^2
\end{equation}

\begin{theorem}\label{infiniteisorad}
Any noncompact conical surface with a finite number of conical singularities  and angles less than $2\pi$ which is flat, complete and of quadratic area growth has infinite isothermal radius. 
\end{theorem}

\begin{proof}
Suppose $M$ is a complete flat conical surface of quadratic area growth, so that there exist $p_1, \dots, p_n$ in $M$ such that near each $p_i$ we can find a coordinate $z$ and a harmonic function $u$ with $g = e^{2u}|z-z_i|^{2\beta_i}|dz|^2$  where $z_i = z(p_i)$ and (\ref{quadratic}) is satisfied. If we smooth out the singularities $p_i$, the resulting surface $M'$ is still noncompact and complete. Moreover, since the cone angles are less than $2\pi$, the curvature can only increase upon smoothing the singularities, thus $M'$ has Gaussian curvature $K \geq 0$. It further follows from the volume comparison theorem of Bishop-Gromov \cite{petersen}   that if the original (singular) surface has quadratic area growth, then any smoothing $M'$ has at most quadratic area growth.

Now fix a smoothing $M'$ of $M$ and let $\tilde{M}$ be its universal cover. Since $M'$ is complete and noncompact, the pullback of the metric on $M'$ to $\tilde{M}$ by the covering map makes $\tilde{M}$ into a simply connected, complete, noncompact surface with Gaussian curvature $K \geq 0$. A complete surface is hyperbolic if it admits a positive Green's function. On the other hand, a theorem of Yau  \cite{yau} (see also \cite{blancfiala}) asserts that a complete surface of nonnegative Gaussian curvature admits no non-constant positive superharmonic functions. Thus $\tilde{M}$ cannot be hyperbolic and by the uniformization theorem, we have $\tilde{M}$ is parabolic, i.e. it is conformally equivalent to the complex plane. 

We actually claim that $M'$ is simply connected, so that $M'$ is parabolic. By Bishop-Gromov again, the universal cover $\tilde{M}$ of $M'$ has at most quadratic area growth. Moreover, since $M'$ has at least quadratic area growth, by Proposition 1.2 in \cite{andersonorbifold} we then have that $M'$ is the quotient of $\tilde{M}$ by a finite group of isometries $\Gamma$. Since $M'$ is smooth, $\Gamma$ must be trivial.

Since smoothing out the singularities doesn't change the topology or the conformal structure, $M$ is also simply connected and parabolic. Therefore, there exists a global coordinate $z$ such that the metric on $M$ has the form $e^{2v}|dz|^2$. Since $M$ has conical singularities, $e^{2v} = e^{2u} \prod |z- z_i|^{2\alpha_i}$. Hence, there exist global coordinates on $M$ for which the metric has the form
\begin{equation}\label{metric}
g = e^{2u} \prod |z- z_i|^{2\alpha_i} |dz|^2
\end{equation} 
where $u$ is harmonic. At this point, we have shown that a noncompact, complete, flat conical surface with cone angles less than $2\pi$ and quadratic area growth is conformally equivalent to the complex plane. Our final claim is that the function $u$ in (\ref{metric}) has to be constant. To see this, define 
\begin{equation}
h = \prod |z-z_i|^{-\alpha_i} g
\end{equation}
where $g$ is the metric in (\ref{metric}). In other words, $h = e^{2u}|dz|^2$. Since $g$ is flat by assumption, the function $u$ is harmonic. Therefore, the Gaussian curvature $K_h$ of $h$ satisfies $K_h= e^{-2u}\Delta u =0$. On the other hand, since the cone angles are less than $2\pi$, $\alpha_i <0$ for all $i$. Therefore, if $\gamma$ is any $C^1$ curve, then far away from the cone points,
\begin{align}\label{lengths}
\int h(\dot{\gamma}(t), \dot{\gamma}(t))^{\frac{1}{2}} dt = \int \prod |z(\gamma(t))- z_i(\gamma(t))|^{\frac{-\alpha_i}{2}}  g(\dot{\gamma}(t), \dot{\gamma}(t))^{\frac{1}{2}} dt \geq \int g(\dot{\gamma}(t), \dot{\gamma}(t))^{\frac{1}{2}}dt
\end{align}
where the inequality follows since $\prod |z(\gamma(t))- z_i(\gamma(t))|^{\frac{-\alpha_i}{2}} >>1$ whenever $ |z(\gamma(t))- z_i(\gamma(t))| >>1$. The assumption that $g$ is complete together with (\ref{lengths}) now imply that $h$ is complete.

 Now, let $F: T_{0}\C \to \C$ be the exponential map of the origin for the metric $h$. By the above arguments, the smooth metric $h$ is flat and complete, so by Cartan-Hadarmard's theorem (see \cite{petersen}, Thm 22) the exponential map $F$ is in fact a diffeomorphism of $\C$ that satisfies 

\begin{equation}\label{exp}
F^*h = e^{2u(0)}|dz|^2
\end{equation}
In other words, 
\begin{align*}
F^* h = e^{u\circ F} |\partial F +\bar{ \partial} F|^2 = e^{u\circ F} (|\partial F|^2 +|\bar{ \partial} F|^2 + 2Re \partial{F}\bar{\partial}{F})= e^{2u(0)}|dz|^2
\end{align*}

The last equality implies that either $\partial{F} = 0$ or $\bar{\partial}{F} = 0$, and since the orientation of the tangent space is the same as the base manifold, we must have that $\bar{\partial}{F} = 0$, so $F$ is holomorphic. A  standard result in complex analysis is that holomorphic diffeomorphisms of the plane are affine linear maps. Given that $F$ preserves the origin (it is the exponential map at the origin), we conclude that it must be of the form $$F(\zeta) = c\zeta$$ for some $c\neq 0$. But then $e^{u\circ F} = e^{u\circ (c\zeta)} = e^{2u(0)}$, so $u$ has to be constant. 

The sequence of arguments above now show that $M$ admits global coordinates for which the metric has the form $$g = C \prod |z- z_i|^{2\alpha_i} |dz|^2$$ from which it follows that the isothermal radius must be infinite. 
\end{proof}

The following is a generalization of $C^{1,\alpha}$ convergence to weighted $\wlp{1}{\alpha}{\gamma}$ convergence of conical metrics on Riemann surfaces. 

\begin{theorem}\label{convergence}
Let $M_i = (\Sigma, g_i, \beta)$ be a sequence of complete, conformal conical surfaces with metrics $g_i$ representing the divisor $\beta$. Let $\{x_i\} \in M_i$ be a sequence of points. Given $\Lambda >0, C>1, \alpha \in (0,1)$, suppose that 
\begin{enumerate}
\item for any $i$, $||K(g_i)||_{0} \leq \Lambda$ away from the cone points
\item there exists $r>0$ such that for any sequence of points $(y_i)$ in $M_i$ there is an isothermal chart $H_i:\Omega_n \to B_0(r)$ where $\Omega_i$ is some open set in $M_i$ and $B_0(r)\subset \C$ such that for any $i$, there exists $\phi_i$ smooth, with $\frac{1}{C} \leq \phi_i(z) \leq C$ such that 
\begin{equation}\label{localrep}
(H_i^{-1})^*g_i = e^{2\phi_i(z)} \prod_{j=1}^{m}|z-z_j|^{2\alpha_j}|dz|^2
\end{equation}
where $z$ are holomorphic coordinates on $B_0(r)$ and $z_j = H_n(p_j)$ correspond to the cone points $p_j$  
\item for $\gamma \in \R^n$, a subsequence of $(H_i^{-1})^*g_i$ converges in $\wlp{1}{\alpha}{\gamma}$ on $B_0(r)$
\end{enumerate}
Then there exists a complete Riemannian manifold $M$ of class $\wlp{2}{\alpha}{}$, there exists a conformal conical metric $g$ of class $\wlp{1}{\alpha}{\gamma}$ and a point $x \in M$ such that the following holds: for any compact domain $D \subset M$ with $x\in D$ there exist, up to passing to  a subsequence, compact domains $D_i \subset M_i$ with points $x_i \in D_i$ and $\wlp{2}{\alpha}{}$ diffeomorphisms $\Phi_n: D \to D_i$ satisfying 
\begin{align}
&\lim_{i\to \infty}(\Phi_i^{-1})(x_i) = x\\
&\Phi_i^*g_i \textit{ converges in } \wlp{1}{\alpha}{\gamma} \textit{ in any chart of the induced } \wlp{2}{\alpha}{} \textit{ complete atlas of } D. 
\end{align}
\end{theorem}

\begin{proof}
The proof is exactly as in the smooth case (see \cite{ hebey, peters}  and \cite{1anderson1} for a summarized version) since conical surfaces are considered to have singularities only in the metric sense, as can be seen from the local coordinate expression in (\ref{localrep}) above. 
\end{proof}

\begin{theorem}\label{compactthm}:
Let $\Sigma$ be a compact Riemann surface without boundary and fix a divisor $\beta= \sum_{j=1}^n \beta_j p_j$ on $\Sigma$ such that  $-1 < \beta_j <0$ for all $j$. If $M_i= (\Sigma, g_i, \beta$) is a sequence of smooth conformal conical metrics on $\Sigma$ representing the divisor $\beta$ such that there exist constants $D_0, \Lambda, v_0>0$ for which 

\begin{enumerate}
\item $diam(M_i) \leq D_0$
\item  there exists $t_0>0$ such that $vol_{g_i}(B(r)) \geq v_0$ for every $r \leq t_0$ 
\item $||K_{g_i}(x)||_{0} \leq \Lambda$ away from the cone points
\end{enumerate}
Then for any $\gamma\in \R^n$, there exists a subsequence of $(g_i)$, $\wlp{2}{\alpha}{}$ diffeomorphisms $F_i:\Sigma \to \Sigma$ and a $\wlp{1}{\alpha}{\gamma}$ conformal conical metric $g$ representing a divisor $\beta'$ such that 
\begin{equation}
||(F^*_i g_i)_{st}- g_{st}||_{1, \alpha; \gamma} \to 0
\end{equation}
as $i\to \infty$, where  $\beta' = \sum_{j=1}^m \beta_j' q_j$ with $-1<\beta'_j \leq 0$, $m\leq n$.
\end{theorem}

\begin{proof}
The first part of the proof is a blow-up argument to show that, under the hypotheses of the theorem, there is a uniform lower bound on the isothermal radius. So to begin, let $D_0, \Lambda, v_0$ be positive constants as in the statement of the theorem and $\alpha \in (0,1)$. Given a conical Riemann surface $(\Sigma, g, \beta)$ and $C, t_0>0$ satisfying 
\begin{align}
||K||_{0} &\leq \Lambda \\
 vol_g (B(r)) \geq v_0 &>0  \text{   } \forall r \leq t_0 \label{cond2}
\end{align}
we show that there exists $r_0 = r_0(\Lambda, v_0)>0$ such that for every $x \in B(r)$ and any $\gamma \in \R^n$
\begin{equation}\label{isolower}
\dfrac{\hrad(x, C, \alpha)}{d_g(x, \partial B(r))}\geq r_0 >0
\end{equation}
where $\hrad(x, C, p)$ is as usual the isothermal radius. 
Indeed, if (\ref{isolower}) does not hold, there exists a sequence of conical metrics $M_i= (\Sigma_i, g_i, \beta)$ representing the same divisor $\beta$ with Gaussian curvature $|K_i|_{0} \leq \Lambda$, there exists a sequence of balls $B_i = B_i(r) \subset M_i$ of radius $r \leq t_0$ , there exists $\gamma \in \R^n$ and there exists a sequence of points $x_i \in B_i$ such that 
\begin{equation}
\dfrac{r_i(x_i)}{d_i(x_i, \partial B_i)} \to 0
\end{equation}
where $d_i = d_{g_i}$ is the induced distance on $M_i$ from $g_i$ and $r_i(x) = r_I(g_i, x)$ is the isothermal radius of the metric $g_i$ at $x$. 

Set $R_i(x) := \frac{r_i(x)}{d_i(x, \partial B_i)}$. By the same arguments as in \cite{hebey, 1anderson1}, we may as well assume the points $x_i$ minimize $R_i(x)$. Rescaling the metrics $g_i$ as 
\begin{equation}
h_i = \dfrac{1}{r_i(x_i)^2}g_i
\end{equation}
we get 
\begin{align}
\lim_{i \to \infty}||K_{(B_i, h_i)}||_{0} &= r_i(x_i)^2 ||K_i||_{0} \leq r_i(x_i)^2\Lambda  \to 0\label{curv} \\
\lim_{i \to \infty}vol(B_i, h_i) &= \infty \label{vol}\\
\lim_{i \to \infty}d(x_i, \partial B_i) &= \infty \label{diameter}
\end{align}
where (2.37) holds away from the cone points. By Lemma \ref{propertiesisorad} (3), the isothermal radius scales as the distance under rescalings of the metric, thus the isothermal radius of the new metrics $h_i$ satisfies
\begin{equation}
r'_i (x_i) := \hrad(h_i, x_i) = 1
\end{equation}
Moreover, for every $y \in B_i$ and for every $i$
\begin{align*}
r'_i(y) = \dfrac{r_i(y)}{r_i(x_i)} \geq \dfrac{d_i (y, \partial B_i)}{d_i(x_i, \partial B_i)} = \dfrac{d'_i(y, \partial B_i)}{d'_i(x_i, \partial B_i)}
\end{align*}
(where the first inequality follows since $R_i(y) \geq R_i(x_i)$ and $d'_i$ is the induced distance from $h_i$). Set 
\begin{equation}
\delta_i := \dfrac{1}{d'_i(x_i, \partial B_i)}
\end{equation}
Then $\lim_{i \to \infty} \delta_i = 0$ and for all $y \in B(x_i, \frac{1}{2\delta_i})$ (the geodesic ball for the metric $h_i$  with center $x_i$) we have 
\begin{equation}
r'_i(y) \geq \frac{1}{2}
\end{equation}
Hence, the isothermal radius of the rescaled metrics is bounded from below and is at most  $1$.
Now we claim $(B_i, x_i, h_i)$ converges in $\wlp{1}{\alpha}{\gamma}$ uniformly on compact sets to a complete (noncompact) manifold $(M, y, h)$. First, the argument above implies that given $R<\infty$, $r'_i(y) \geq \frac{1}{2}$ on $B(x_i, R)$ for $i$ large enough. Thus given $R < \infty$ and a sequence $(q_i)$ in $B(x_i, R)$ we can find isothermal charts $H_i:  \Omega_i \to B_0 (\frac{1}{2\sqrt{C}}) $ centered at $q_i$ such that  
\begin{align}
(H_i^{-1})^*h_i &= e^{2\phi_i(z_i)}\prod_{j=1}^{m_i} |z_i - z_i^j|^{2\alpha_j}|dz|^2
\end{align}
where $\phi_i(z_i): B_0 (\frac{1}{2\sqrt{C}})\to \R$ are smooth and bounded, the integers $m_i = m_i(q_i)$ and the real numbers $\alpha_j$ satisfy $1\leq m_i \leq n$, $-1 < \alpha_j \leq 0$ for all $i= 1, \dots $ and $1 \leq j \leq m_i$. As before, $z_i$ is a holomorphic coordinate on $B_0 = B_0(\frac{1}{2\sqrt{C}})$ and $z^j_i = H_i(p_j)$ ($p_j$ are cone points). Moreover, 
\begin{align}
&\dfrac{1}{C} \leq \phi_i(z_i) \leq C\\
&||\phi_i(z_i)||_{1,\alpha; \gamma} \leq F(C)  \label{thisone}
\end{align}
with $F$ depending only on $C$. By (\ref{thisone}) we have that $(\phi_i(z_i))$ are bounded in $\wlp{1}{\alpha}{\gamma}$ on the ball $B_0$, so after passing to a subsequence we can assume they converge weakly in $\wlp{1}{\alpha}{\gamma}$ to some $\phi$ by the Arzela-Ascoli Theorem. In particular, the metrics $(H_i^{-1})^*h_i =: h'_i$ converge weakly in $\wlp{1}{\alpha}{\gamma}$  on $B_0$.  We claim that in fact the metrics $h'_i$ converge strongly in $\wlp{1}{\alpha}{\gamma}$.  Indeed, from Lemma \ref{gaussiancurvature}, we have that the Gaussian curvature $K_i$ of the metrics $h_i$ in the coordinates $z_i$ is given by 
\begin{equation}
K_i(z_i) = \dfrac{e^{-2\phi_i(z_i)}\Delta \phi_i(z_i)}{\prod_{j=1}^{m_i}|z_i-z^j_i|^{2\alpha_j}}
\end{equation}
As we observed in the proof of Lemma \ref{uppersemicont}, we can assume $z_i^j = z^j$ and $m_i = m$ are independent of $i$ by passing to a subsequence.
By $(\ref{curv})$, $||K_{(B_i, h_i)}||_{0}  \to 0$ away from the cone points, so that $||(H^{-1}_i)^*K_i||_{0} = ||K_i(z_i)||_{0} \to 0$ as $i \to \infty$ for all $z_i \neq z^j \in B_i$. Now, the limit $\phi(z)$ of the sequence $(\phi_i(z_i))$ solves 
 \begin{equation}\label{curva}
 0 = \dfrac{e^{-2\phi(z)}\Delta \phi(z)}{\prod_{j=1}^{m}|z-z^j|^{2\alpha_j}}
 \end{equation}
weakly. Applying standard elliptic estimates to (\ref{curva}), we get $\phi(z)$ is actually smooth. In fact, using the same arguments as in the proof of  Lemma \ref{uppersemicont} (second paragraph following (\ref{A1bounds})), we get that the convergence is actually in the \textit{strong} $\wlp{1}{\alpha}{\gamma}$ topology, hence the metrics $h'_i$ converge strongly in $\wlp{1}{\alpha}{\gamma}$.

It now follows from the arguments used to prove Theorem \ref{convergence} (see for instance Proposition 12 in \cite{hebey}) that there exists a $\wlp{2}{\alpha}{}$ manifold $M$,  $y \in M$ and  a $\wlp{1}{\alpha}{\gamma}$ conformal conical metric $h$ on $M$ such that for any compact domain $D \subset M$ with $y \in D$ and after passing to a subsequence, there exist compact domains $D_i \subset B_i$ and $y_i \in D_i$ and  $\wlp{2}{\alpha}{}$ diffeomorphisms $\Phi_i: D \to D_i$ such that 
\begin{align}
\lim_{i\to \infty}(\Phi_i^{-1})(y_i) &= y\\
||(\Phi_i^{-1})^*h_i - h||_{1,\alpha;\gamma} &\to 0 \text{ in } D \label{converg}
\end{align}
where (\ref{converg}) is in the sense that in any chart of the complete induced atlas on $D$ the components of $(\Phi_i^{-1})^*h_i$ converge in $\wlp{1}{\alpha}{\gamma}$ to the components of $h$. 

We now claim the pointed limit $(M,h)$ is flat, conical and complete. First, the completeness follows from (\ref{diameter}). To see $(M, h)$ is flat, given a compact domain $D$ in $M$, set $\hat{h}_i:= \Phi^*_i h_i$ and for a given $x \in D$, let $U_i: B_x(r) \to \C$, $r>0$, be an isothermal coordinate chart for $\hat{h}_i$ satisfying $A1, B1$ of Definition \ref{defharmradius}.  The convergence of the $\hat{h}_i$ implies convergence in $\wlp{2}{\alpha}{}$ of the charts $U_i$ to a limiting chart $U:B_x(r) \to \C$, by the same arguments as in Lemma \ref{uppersemicont}. If we write $(U^{-1})^*h = e^{2\psi(z)}\prod_{j=1}^{m}|z-z_j|^{2\alpha_j} |dz|^2$, then going back to (2.48)  we have 
\begin{equation}\label{this}
 0 = \dfrac{e^{-2\psi(z)}\Delta \psi(z)}{\prod_{j=1}^{m}|z-z_j|^{2\alpha_j}}
\end{equation}
It follows from standard elliptic estimates then that the limit metric $h$ is smooth on $\Sigma^*$. Since (\ref{this}) is also the equation for the Gaussian curvature of $h$, we also have that it is flat. Moreover, assumption (2) in the Theorem along with the Bishop-Gromov volume comparison Theorem imply that 
\begin{equation}
\dfrac{vol(B(s))}{V^\Lambda(s)}\geq \dfrac{v}{V^\Lambda(t_0)}
\end{equation}
for all $s \leq t_0$, where $V^\Lambda$ is the volume of a geodesic ball in constant curvature $\Lambda$ and $vol(B(s))$ is taken in the $g_i$ metric. Using scaling properties of volume and the convergence in $\wlp{1}{\alpha}{\gamma}$, we have 
\begin{equation}
\dfrac{vol_M (B(s))}{s^2} \geq v' >0
\end{equation}
for all $s>0$. On the other hand, since $(M, h)$ is flat with cone angles less than $2\pi$, the volume of a ball of radius $s$ measured in the metric $h$ has to be less than the volume of a ball of radius $s$ measured with respect to the standard metric on $\R^2$, hence 
\begin{equation}
\dfrac{vol_M (B(s))}{s^2} \leq v'
\end{equation}
and we conclude that the volume growth must be exactly quadratic. 

To obtain a contradiction, observe that since the $h_i$ converge to $h$ in $\wlp{1}{\alpha}{\gamma}$, by  Lemma \ref{uppersemicont}, we  have that 
\begin{equation}
\hrad(C', p, y) \leq \lim_{i\to \infty} \hrad(C, p, y_i) 
\end{equation}
for some $C' < C$, but by construction $\hrad(C, p, y_i) = 1$, while a flat, noncompact, complete  surface with finitely many conical singularities and of quadratic area growth has infinite isothermal radius for any $C'$, as follows from Theorem \ref{infiniteisorad}.

Therefore, there is a uniform lower bound on the isothermal radius, which in turn allows us to apply Theorem \ref{convergence} to conclude the existence of a limit with the desired properties. In the notation of Theorem \ref{convergence}, with $D = B_x(R)$, $R> D_0$, we get that for $i$ large enough, $D_i = M_i$, and up to passing to a subsequence there exist diffeomorphisms $\Phi_i: M \to M_i$ such that $\Phi_i^*g_i$ satisfies conditions (2.30-2.31) of Theorem \ref{convergence}. In particular, $M$ has a smooth structure coming for instance from one of the diffeomorphisms $\Phi_i$ with $M_i$. 
\end{proof}

\section{Conformal Conical Metrics on the 2-sphere}
In this section we are concerned with  conformal conical metrics on the 2-sphere with angles less than $2\pi$ and at least three conical singularities. In the language of Definition \ref{conicalconformaldef}, let $n \geq 3$, $\beta = \sum_{i=1}^n \beta_i p_i$ be a divisor on $S^2$ such that $-1< \beta_i <0$ for each $i =1 ,\dots n$ and $g_\beta$ be a conical metric on $S^2$ representing the divisor $\beta$. We begin with some results concerning the conformal geometry of $(S^2, \beta)$.

Define $Iso (S^2, g_\beta)= \{ \phi \in Diff(\Sigma): \phi^* g_\beta = g_\beta\}$, which forms a group under composition. 
A diffeomorphism $\psi: S^2 \to S^2$ is called a conformal transformation of $(S^2, g_\beta) $ if $\psi^*(g_\beta) \in [g_\beta]$, i.e. if there exists a function $u: S^2 \to \R$ which is smooth and positive and such that 
	$$\psi^*g_\beta = e^{2u}g_\beta$$
It is clear that the set of all conformal transformations forms a group under composition. We denote it by $Conf(S^2, g_\beta)$. Observe that if $\psi\in Diff(S^2)$, then $\psi^*g_\beta$ is always a conical metric with the same number of conical singularities as $g_\beta$. 
 	Suppose $\phi \in Conf(S^2, g_\beta)$. Then by definition, $$\phi^* g_\beta = \eta^2_\phi g_\beta$$
 	where $\eta_\phi = |D\phi|>0$. Now suppose $g= e^{2u}g_\beta$. We have $$\phi^* g = e^{2(\phi^*u)}\phi^*g_\beta = e^{2(u\circ \phi + \log \eta_\phi)}g_\beta$$
 	Hence the conformal group $Conf(S^2, g_\beta)$ acts on the conformal factors $u \in \dom$ by 
 	\begin{equation}\label{confactiononfactors}
 		(\phi, u) \to u\circ \phi + \log \eta_\phi
 	\end{equation}
 	
 On the other hand, $Conf(S^2, g_\beta)$ acts on the functions $K\in \cdom$ by precomposition, i.e. 
 	\begin{equation}\label{confactiononcurvature}
 		(\phi, K) \to K \circ \phi
 	\end{equation}

 The group of diffeomorphisms of the sphere $S^2$ which are $\wlp{k-2}{p}{}$ is denoted by Diff$^{k-2, p}$ and acts on the curvature functions again by precomposition, as in (\ref{confactiononcurvature}). 

\begin{theorem}\label{finiteconformal}
Suppose $\beta = \sum_{i=1}^n \beta_i p_i$ is a divisor on $S^2$ and $g_{+1}$ is the round metric of curvature $1$ on $S^2$. If $n\geq 3$ and $g_\beta$ is a conical metric on $(S^2, g_{+1})$ representing the divisor $\beta$, then $Conf(S^2, g_\beta)$ is finite. Moreover, if there exists $i,j,k$ distinct, such that $\beta_i, \beta_j, \beta_k$ are all distinct, then the conformal group $Conf(S^2, g_\beta) = 0$. 
\end{theorem}

\begin{proof}
To begin, observe that every conformal transformation $\psi \in Conf(S^2 - \{p_1, p_2, \dots p_n\}, g_{+1})$ has an extension $\tilde{\psi} \in Conf(S^2, g_{+1})$. This follows from the fact that any conformal map $\psi\in Conf(S^2 - \{p_1, p_2, \dots p_n\}, g_{+1})$ can be viewed as a biholomorphism $\C - \{q_1, \dots, q_{n-1}\} \to \C - \{q_1, \dots, q_{n-1}\} $ after conjugation with the stereographic projection from, say, $p_n$. 

To be precise, let $\sigma_{p_n}$ be the sterographic projection $S^2 - \{p_n\} \to \C$ and suppose $q_i = \sigma_{p_n}(p_i)$ for $i = 1, \dots n-1$.  The restriction of $\sigma_{p_n}$ to the punctured sphere $S^2- \{p_1, \dots p_n\}$ gives a diffeomorphism with $\C - \{q_1, \dots q_{n-1}\}$. In particular, if $\psi \in Conf(S^2 - \{p_1, \dots, p_n\}, g_{+1})$, then 
$$\bar{\psi}:= \sigma^{-1}_{p_n}\circ \psi \circ \sigma_{p_n}: \C - \{q_1, \dots q_{n-1}\} \to \C - \{q_1, \dots q_{n-1}\}$$
is a biholomorphism. We claim now that the points $q_i$ are removable singularities for $\bar{\psi}$. Indeed, if $\bar{\psi}$ had an essentially singularity at any of the points $q_i$, then by Picard's Theorem, $\bar{\psi}$ would take on all possible values with at most one exception on any neighborhood of $q_i$, infinitely often, but this would contradict injectivity. The second observation is that $\psi$ has at worst one pole of order one. Again, any higher order pole is excluded because of injectivity. If there were two simple poles at say $q_i, q_j$, then $\bar{\psi}(q_i)= \infty$, $\bar{\psi}(q_j)= \infty$, but since $\psi$ is an open map, it would map a punctured neighborhood of $q_i$ to a neighborhood of $\infty$ and a punctured neighborhood of $q_j$ to a neighborhood of $\infty$. The fact that these two neighborhoods must intersect contradicts injectivity once again.

It then follows that $\bar{\psi}$ extends to a biholomorphism of the Riemann sphere, hence it corresponds to a conformal map $\tilde{\psi}$ of $S^2(1)$.

Now, observe that if $p_1, p_2 \dots, p_n$ are cone points for the conical metric $g_\beta$, then 
$$Conf(S^2 - \{p_1, \dots, p_n\}, g_\beta) = Conf(S^2 - \{p_1, \dots, p_n\}, g_{+1})$$
This follows from the fact that $g_\beta$ is conformal to $g_{+1}$ by definition.  

Finally, we conclude that the group $Conf(S^2 - \{p_1, \dots, p_n\}, g_{+1})$ is finite if $n\geq 3$. Indeed, if $\psi \in Conf(S^2- \{p_1, \dots p_n\}, g_{+1})$, then its extension $\tilde{\psi}$ to a conformal map of the round sphere fixes the set $\{p_1, \dots p_n\}$. Suppose $\tilde{\psi}(p_1) = p_i, \tilde{\psi}(p_2) = p_j $, $\tilde{\psi}(p_3) = p_k$, where $i, j, k \in (1, \dots n)$ are all distinct. Since $\tilde{\psi}$ is a Mobius transformation, its values are uniquely determined after specifying the image of the points $p_1, p_2, p_3$. Since there are only finitely many choices for $p_i, p_j, p_k$, the collection of extensions of $\psi \in Conf(S^2- \{p_1, \dots p_n\}, g_{+1})$ is finite. In particular, the group $Conf(S^2 - \{p_1, \dots, p_n\}, g_{+1})$ is finite.

 Now, let $supp (\beta) = \{p_1, \dots p_n\}$. Define $F: Conf (S^2, \beta, g)  \to Conf (S^2 - supp(\beta), g)$ by $F(\phi) = \phi\big|_{S^2 - supp(\beta)}$. Since any $\phi \in Conf (S^2, \beta, g) $ fixes the set $supp(\beta)$, we have that $F$ is a surjective group homomorphism. Moreover, 
$$ker(F)= \{ \phi \in Conf(S^2, \beta, g): F(\phi) =\phi|_{S^2 - supp(\beta)}= Id|_{S^2 - supp(\beta)}\} $$
The only freedom is in where the points $p_1, \dots p_n$ are sent, and we know $\phi$ fixes them on $(S^2, \beta, g)$. Hence we have 
$ker(F)$ is isomorphic to a subgroup of $S_n$, the symmetric group on $n$ elements. Finally, since $Conf (S^2 - supp(\beta), g)$ is finite for $n\geq 3$, and  $ker(F)$ is also finite, we must have $Conf (S^2, \beta, g)$ is finite.
If there are three distinct angles, any conformal map has to fix them, but every conformal map of the unit disk fixing three points is the identity. So  the conformal group must be trivial in this case. This concludes the proof of the theorem. 
\end{proof}

\begin{remark}
The condition that $n\geq 3$ is only sufficient. In the examples below we show that there is a metric on $S^2(1)$ with two conical singularities and noncompact conformal group. 
\end{remark}

\subsection{Examples}

\textit{Example 1: The football}. Let $\Sigma = S^2(1)/\Z_k$ where $\Z_k$ acts by rotations. The quotient is known as the American football and it is a topological sphere with two conical singularities, each of angle $\frac{2\pi}{k}$. If $\bar{g}$ is the induced metric on the quotient, i.e. $\pi^*\bar{g} = g_{+1}$, where $\pi:S^2 \to S^2/\Z_k$ is the quotient map, then $Conf(S^2/\Z_k, \bar{g})$ is noncompact. Indeed, let $\phi_\lambda: S^2 \to S^2$ be $\phi_\lambda(x) = \sigma_p^{-1} \circ \delta_\lambda \circ \sigma_p$ where $\delta_\lambda(x) = \lambda x$ ($\lambda >0$) is a dilation of $\R^3$ and $\sigma_p$ is the stereographic projection from $p$ (the composition is extended to the whole sphere by sending $p \to p$) . The action of $\Z_k$ on $S^2$ can be viewed as an action of $\Z_k$ on $\C$ after identifying $S^2-{p}$ with the complex plane via the stereographic projection. From this point of view, for each element $[m] \in \Z_k$ we get a map $\psi_m (z) = \zeta^m\cdot z$, where $\zeta$ is a $kth$ root of unity. Then one can check that 
$$\phi_\lambda \circ \psi_m = \psi_m \circ \phi_\lambda$$ for every $[m] \in \Z_k$. In particular, the map $\phi_\lambda$ descends to the quotient and it will be a conformal map of $(S^2/\Z_k, \bar{g})$. 

\textit{Example 2. Variation on the Football.} Another way to obtain the American football of Example 1 is by cutting out two neighborhoods of say, the north and south pole and gluing back two different cones, e.g we can replace a neighborhood $U_1$ of the north pole by a quotient of the disk $D/\Z_n$ and a neighborhood $U_2$ of the south pole by the quotient $D/\Z_m$. When $n=m$, upon choosing appropriate gluing maps and metrics on the cone pieces, this space can be realized as the quotient $S^2/Z_n$ of Example 1. In the case when $n\neq m$,  there is no metric of constant curvature with conical singularities at the poles (\cite{troyanov}). 

\textit{Example 3. Double of a Spherical Triangle.} Let $T$ be a spherical triangle in $S^2(1)$ with angles $\alpha = \frac{2\pi}{n}, \beta = \frac{2\pi}{m}, \gamma =\frac{2\pi}{p}$. Construct the double of $T$ by identifying $T$ with itself via the identity. The resulting space $M$ is a topological sphere with $3$ conical singularities of angles $2\alpha, 2\beta, 2\gamma$. The conformal group is the dihedral group $D_6$. 
 
 \subsection{Prescribing Gaussian Curvature on the 2-sphere}

 For $\gamma = (\gamma_1, \dots \gamma_n)$, let $\dom$ be the Banach space of $\dom$ functions $u: S^2 \to \R$ considered as conformal factors of $g= e^{2u}g_\beta$ and let $\cdom$ be the Banach space of $\cdom$ functions $K$.  Our main goal is to study the image of the \textit{curvature map} $\pi$, defined to be the map $\dom \to \cdom $ sending $$u \mapsto K_g  $$
As before, if $g$ is a conformal conical metric on $(S^2, g_{+1})$, then 
 $$g = e^{2u}g_\beta = e^{2u}\rho^{2\beta} g_{+1}$$ 
 where $\rho$ is a radius function as in Definition \ref{radiusfunction}. 
The Gaussian curvature of $g$ is then 
\begin{align*}
K_g = K(e^{2u}g_\beta)=e^{-2u} (K_\beta - \Delta_\beta u)
\end{align*}
where $\Delta_\beta$ is the Laplacian with respect to the conical metric $g_\beta$ and $K_\beta$ the Gaussian curvature of $g_\beta$. One can compute 
\begin{equation}
\Delta_\beta = \rho^{-2\beta}\Delta_{+1}
\end{equation}

\begin{equation}
K_\beta = \rho^{-2\beta} (1 - \beta\Delta_{+1}\log \rho ) 
\end{equation}

Observe that the function $\beta \Delta_{+1} \log \rho$ is defined to be $\beta_i \Delta_{+1} \rho$ in a neighborhood of the cone point $p_i$ and vanishes identically away from the cone points (since $\rho \equiv 1$).

Recall that a $C^1$ map $F: \mathcal{B}_1 \to \mathcal{B}_2$ is a Fredholm map between Banach manifolds $\mathcal{B}_i$ if the differential $$D_{u}F(h) := \dfrac{d}{dt}F(u+th)\res_{t=0}= \lim_{t \to 0} \dfrac{F(u+th) - F(u)}{t}$$ is a Fredholm operator at each $u \in \mathcal{B}_i$. As it is well-known,  Fredholm maps between Banach spaces are bounded linear operators characterized by having finite-dimensional kernel and cokernel. The index of a Fredholm operator is defined as 
$$ind(F) =dim(Ker F)- dim(coKer F)$$
Moreover, the index of the Fredholm map $F$ is defined to be the index of its differential, which is independent of the choice of $u$. For more on the  theory of Fredholm maps on Banach manifolds, see \cite{elworthy, nirenberg}.

\begin{theorem}\label{fredholmmap}
Let $(\Sigma,g, \beta)$ be a conical surface with $g$ representing the divisor $\beta = \sum_{j=1}^n \beta_j p_j$. If $\gamma= (\gamma_1, \dots \gamma_n) \in \R^n$ satisfies  $\gamma_i >0$ and $\gamma_i \neq \frac{m}{\beta_j}$ for any $i,j \in (1, \dots n)$, where $m$ is an integer, then the curvature map $\pi$ is Fredholm of index 0.
\end{theorem}	

\begin{proof}
Fix $u \in \dom$. If we set $g = e^{2u}g_\beta$, then as observed above
\begin{equation}
\pi(u) = e^{-2u}(K_{\beta}- \Delta_{\beta} u)
\end{equation}

Let $h \in \dom$, thought of as the tangent space to $\dom$ at $u$. Then 
\begin{align*}
D_u\pi(h) &= \dfrac{d}{dt}\res_{t=0}\pi(u+th)\\
&= \dfrac{d}{dt}\l e^{-2(u+th)}(K_{\beta} - \Delta_{\beta}(u+th))\r\res_{t=0}\\
&= -2hK_g - e^{-2u}\Delta_{\beta}h
\end{align*}
with $K_g = e^{-2u}(K_{\beta} - \Delta_{\beta}u)$, which is the Gaussian curvature of $g$. At $u = 0$, we get $K_g = K_{\beta}$, hence 
\begin{align}
 D_0 \pi (h) &= -2hK_{\beta}- \Delta_{\beta} h \\
&=-\rho^{-2\beta}(2h (1- \beta \Delta_{+1} \log \rho)) + \Delta_{+1} h)
\end{align}

If we let $a:= 2(1 - \beta \Delta_{+1} \log \rho)$, then 
\begin{equation}\label{Lbeta}
 -\LL{\beta}(h):= D_0 \pi (h) = \rho^{-2\beta} (ah + \Delta_{+1} h)
\end{equation}
It is known that if $\gamma= (\gamma_1, \dots \gamma_n) \in \R^n$ satisfies $\gamma_i \neq \frac{m}{\beta_j}$ for any $i,j \in (1, \dots n)$, where $m$ is an integer, then the linear operator $\LL{\beta}: \wlp{k}{p}{\gamma} \to \wlp{k-2}{p}{\gamma-2}$ is Fredholm \cite{mazzeoweiss, behr}. We further claim $\LL{\beta}$ is formally self-adjoint.  Observe there is a natural inner product on $(S^2, \beta)$ given by 
\begin{equation}
\langle u, v \rangle = \int_{S^2} u\cdot v \rho^{2\beta} dV_{+1}
\end{equation}

Thus for all $u, v \in \dom$ we have 
\begin{align}
\langle v, \LL{\beta}u\rangle&= \int_{S^2} v(\rho^{-2\beta}(\Delta_{+1} u + au)) \rho^{2\beta} dV_{+1} \\
&= \int_{S^2} v(\Delta_{+1} u + au)  dV_{+1}\\
& = \int_{S^2} v (\Delta_{+1} u + au) dV_{+1}\\
& = \int_{S^2} u (\Delta_{+1} v + av) dV_{+1}\\
&=  \int_{S^2} u \rho^{-2\beta}(\Delta_{+1} v + av) \rho^{2\beta} dV_{+1}\\
&= \langle L_\beta v, u \rangle
\end{align}

The integration by parts in (3.12)-(3.13) needs some justification, so fix $R>0$ small enough and let $B_R(p_i)$ be a geodesic ball of radius $R$ around the cone point $p_i$. Let $S_R(p_i)$ denote the circle of radius $R$ with center $p_i$, $\partial_{\nu} u$ denotes the normal derivative of $u$, and $dS$ the volume element of $S_R(p_i)$. We then have
\begin{align*}
\int_{S^2- \{p_1\dots p_n\}} v \Delta_{+1} u \quad dV_{+1} & = \lim_{R\to 0}  \int_{S^2 - \cup_{i=1}^n B_R(p_i)} v\Delta_{+1} u  \quad dV_{+1}\\
&= \lim_{R\to 0} \l\int_{S^2 - \cup_{i=1}^n B_R(p_i)} u\Delta_{+1} v  \quad dV_{+1}  + \int_{\cup_{i=1}^n S_R(p_i)} u\partial_{\nu}v - v\partial_{\nu}u \quad dS \r\\
&= \int_{S^2 - \{p_1, \dots p_n\}} u\Delta_{+1} v  \quad dV_{+1} + \lim_{R\to 0} \sum_{i=1}^n \int_{S_R(p_i)} u\partial_{\nu}v - v\partial_{\nu}u \quad dS
\end{align*}

Now, since $u \in \wlp{k}{p}{\gamma}(S^2 - \cup_{i=1}^n B_R(p_i))$ for all $R>0$ small enough, it follows that $\partial_\nu u \in \wlp{k-1}{p}{\gamma-1}(S^2 - \cup_{i=1}^n B_R(p_i))$, i.e. there is a $C>0$ such that $||u||_{C^{l,\alpha}_\gamma} \leq C$. In particular, $||\partial_\nu u||_{C^{l-1, \alpha}_{\gamma-1}} \leq C$.  It follows from the definition of these norms (see 2.14) that \[\sup_{x\in S^2 - \{p_1, \dots p_n\}}\rho^{-\gamma}|u| \leq C\] and \[\sup_{x\in S^2 - \{p_1, \dots p_n\}}\rho^{-(\gamma -1)}|\partial_\nu u| \leq C_1\] 
and similarly for $v$. Therefore 
\begin{align}
\bigg|\int_{S_R(p_i)} u\partial_{\nu}v - v\partial_{\nu}u \quad dS \bigg| &\leq \int_{S_R(p_i)} |u||\partial_\nu v |dS+ \int_{S_R(p_i)} |v| |\partial_\nu u| dS\\
&= \int_{S_R(p_i)} \rho^{-\gamma}|u|\rho^{-\gamma+1}|\partial_\nu v | \rho^{2\gamma -1}dS+ \int_{S_R(p_i)} \rho^{-\gamma}|v| \rho^{-\gamma +1}|\partial_\nu u|\rho^{2\gamma -1} dS\\
&\leq 2C'\int_{S_R(p_i)} \rho^{2\gamma -1} dS\\
&\leq C(\delta) R^{2\gamma}
\end{align}

 Thus, provided $\gamma >0$, taking the limit as $R \to 0$, we see that the boundary terms disappear, as wanted. 
This concludes the proof that the map $\LL{\beta}$ is formally self-adjoint and Fredholm, from which it follows that the map $\pi$ is a Fredholm map of index $0$. 
\end{proof}

 \subsubsection{Properness of the map $\pi$}

\begin{theorem}\label{properness}
 	Let $\mathcal{C_+}$ be the subspace of $\cdom$ consisting of positive curvature functions $K$.  Define $\mathcal{U} = \pi^{-1} (\mathcal{C}_+)$. If $\gamma>0$, then the map $\pi_0: \mathcal{U} \to \mathcal{C_+}$ defined as the restriction of $\pi$ to $\mathcal{U}$ is proper. 
\end{theorem}

\begin{proof}
If $K_i \to K \in \mathcal{C}_+$, then the sequence $K_i$ is bounded in $\wlp{k-2}{\alpha}{\gamma-2}$.  In particular, $K_i$ is bounded in $C^0_{\gamma-2}$. Hence there exist a constant $K$ such that $||K_i||_{0} \leq K$ away from the cone points. 
Moreover, since the Euler characteristic is positive, we have using Gauss-Bonnet (\ref{gauss bonnet}) 
	\begin{equation}\label{n=3 volume bound}
  2\pi \chi(S^2, \beta) = \int_{S^2} K_i dvol_{g_i} \leq K_0 \cdot area(S^2, g_i)\\
	\end{equation}
so we get 
 $$area(S^2, g_i) \geq \dfrac{2\pi \chi(S^2, \beta)}{K_0}>0$$
On the other hand, Myers' theorem (\cite{petersen}) implies that the diameter of each conical surface is finite (since curvature is assumed positive). Furthermore, since the sequence $K_i$ converges in $\mathcal{C}_+$ to a positive function $K$, we can find constants $D_0, v_0$ such that 
\begin{enumerate}
	\item $vol(g_i)\geq v_0$
	\item $diam(g_i) \leq D_0$
\end{enumerate}
for all $i$. Under these bounds we can now directly apply Theorem \ref{compactthm} to conclude that there exists a sequence of diffeomorphisms  $F_i: S^2 \to S^2$ such that 
\begin{equation}
(F_i^*g_i)_{rs} \to (g_\infty)_{rs}
\end{equation}
in $\wlp{1}{\alpha}{\gamma}$, where $g_\infty$ is a conical metric on $S^2$ with $m$  conical singularities of angles $0 < \theta \leq 2\pi$. 
By passing to a subsequence if necessary, we may assume that the $F_i$ are orientation preserving. On the other hand, since $g_\infty$ is the limit of a sequence of conical metrics in the same conformal class, we claim there exists a diffeomorphism $\psi : S^2 \to S^2$ and a smooth, positive function $u:S^2 \to \R$ such that $\psi^*g =e^{2u}g_\beta$. To see this, suppose $\mathfrak{C}$ denotes the space of conformal classes on the punctured sphere $S^2 - \{p_1, \dots, p_n\}$, i.e., two smooth  (incomplete) metrics $h_1, h_2$ on $S^2 - \{p_1, \dots, p_n\}$ represent the same point in $\mathfrak{C}$ if there exists a positive smooth function $u$ such that $h_1 = e^{2u}h_2$. The group $Diff_+$ of orientation preserving diffeomorphisms of $S^2$ acts on $\mathfrak{C}$ by 
	$$(\psi, [h]) \to \psi^*[h] = [\psi^*h]$$
	Since it is not true in general that $[\psi^*h] = [h]$, we consider the moduli space $\mathcal{M} = \mathfrak{C}/Diff_+$. This space corresponds to Teichmuller space \cite{petri, mazzeoweiss}, since the mapping class group of the sphere is trivial (so every orientation preserving diffeomorphism is isotopic to the identity). Observe then that every element of the sequence $g_i$, when considered as smooth metrics on $S^2 - \{p_1, \dots p_n\}$,  corresponds to the same point in $\mathfrak{C}$ , namely the conformal class $[g_\beta]$. If $\psi_i$ is a sequence of orientation preserving diffeomorphisms, then $\psi_i^*[g_i] = [g_i]$ as elements of $\mathcal{M}$, hence the classes $\psi_i^*g_i$ and $g_\beta$ correspond to the same point in $\mathcal{M}$.  Since the topology of Teichmuller space is Hausdorff \cite{petri}, the sequence $\psi_i^*g_i$ is the constant sequence $[g_\beta] \in \mathcal{M}$. Thus the limit of the $\psi^*_ig_i$ must be in the conformal class of $g_\beta$ modulo diffeomorphisms, i.e. there exists a diffeomorphism of $S^2$ such that the limit $g$ satisfies 
	$$\psi^*g = e^{2u}{g_\beta}$$
for some positive smooth function $u$ on $S^2$, as claimed. 
In particular, there is a diffeomorphism $\Psi$ of $S^2$ such that $\Psi^*g_\infty = e^{2u}g_\beta$. After precomposing $F_i$ with $\Psi^{-1}$, we may as well assume that we have a sequence $F_i$ of diffeomorphisms of $S^2$ such that 
\begin{equation}\label{diffeo}
F_i^*(g_i) = F_i^*(e^{2u_i}g_\beta) \to e^{2u}g_\beta
\end{equation}
in $\wlp{1}{\alpha}{\gamma}$, where $u$  is some positive smooth function on $S^2$. 
As in the arguments preceding Proposition 2.5 in \cite{andersonnirenberg}, one now has that the $F_i$ converge on a subsequence to the identity modulo the action of the conformal group, i.e. there exist conformal maps $\phi_i \in Conf(S^2, g_\beta)$ such that $\phi_i^{-1}\circ F_i$ converge to the identity. It  follows from Proposition 2.5 in \cite{andersonnirenberg}  that the functions $\phi_i^*u_i$ are uniformly bounded in $\wlp{1}{\alpha}{\gamma}\cap W^{2,p}_\delta$, for $\gamma >0$ and some $\delta \leq \gamma$ (see \cite{behr} for a definition of weighted Sobolev spaces). Observe that nothing really changes in the presence of conical singularities since the diffeomorphisms $F_i$ are still quasiconformal when restricted to the punctured sphere $S^2 - \{p_1, \dots , p_n\}$ (see \cite{oli}).
 By the Arzela-Ascoli theorem, the uniform bound on the sequence $\phi_i^*u_i$ implies convergence on a subsequence to a limit in $C^{1,\alpha}_\gamma$. Moreover, since the conformal group $Conf(S^2, g_\beta)$ is finite, we actually have that $\{u_i\}$ themselves (sub)converge to a limit $u  \in C^{1, \alpha}_{\gamma}$ which satisfies 
\begin{equation}\label{definingequation}
\Delta_{g_\beta}u = K_\beta - Ke^{2u}	
\end{equation}
weakly. Given that the Gaussian curvatures $K_i$ of the metrics $g_i$ are assumed to be in $ \wlp{k-2}{\alpha}{\gamma-2}$, a bootstrapping argument using Proposition 2.7  in \cite{behr} implies $u \in \wlp{k}{\alpha}{\gamma}$. This then completes the proof that the map $\pi_0$ is proper. 
\end{proof}

\subsubsection{Degree Computations}
We conclude this section with a result providing sufficient conditions for a function $K$ to arise as the Gaussian curvature of a conformal conical metric on $S^2$ having at least three conical singularities and angles less than $2\pi$. As mentioned in the introduction, a necessary condition for the existence of a \textit{constant} curvature conical metric on $S^2$ having at least three conical singularities and angles less than $2\pi$ is 

\begin{equation} \label{troyanovcondition1}
\sum_{i\neq j} \beta_i < \beta _j,  \textit{  for all  } j
\end{equation}

In fact, Luo-Tian have shown in \cite{luotian} that if the generalized Euler characteristic is positive, then this condition is sufficient and necessary for uniqueness and existence. Under these assumptions, we can now compute the degree of the curvature map given our previous results. Recall that if $F$ is a proper Fredholm map of index $0$ between open subsets of Banach spaces, one can define its \textit{degree} by the formula 
	$$deg(F) = \sum_{x\in F^{-1}(y)} sign(D_xF)$$
where $y$ is any regular value of $F$ and the sign is $\pm$ according to whether $D_xF$ preserves or reverses orientation. By definition, $y$ is a regular value if $D_xF$ is an isomorphism for all $x\in F^{-1}(y)$. In particular, points with empty preimage are always regular values. We refer the reader to \cite{nirenberg} for more on the degree theory of Fredholm maps on Banach manifolds. 

Let $\mathcal{C} = \mathcal{C}_+ \cap C^{k-2, \alpha}_{\gamma-2}$, where $\gamma_i >0, \gamma_i \neq \frac{m}{\beta_j}$ for any $(i,j) \in (1, \dots n)$. Recall that the restrictions on $\gamma$ guarantee that the curvature map is proper and Fredholm of index $0$. We have,

\begin{theorem}
Suppose $n \geq 3$, and $\beta = \sum_{i=1}^n \beta_i p_i$ is a divisor on $S^2$ satisfying the Troyanov condition (\ref{troyanovcondition1}) and there exists $i,j,k$ distinct for which $\beta_i , \beta_j, \beta_k$ are all distinct. Assume $\chi(S^2, \beta)>0$ and let $g_\beta$ be the unique conical metric on $S^2$ representing the divisor $\beta$ of Gaussian curvature $K_\beta = 1$. Then a function $K$ on $S^2$ is the Gaussian curvature of a metric $g$ conformal to $g_\beta$ if and only if $K \in \mathcal{C}$.
\end{theorem}

\begin{proof}
Suppose $K \in \mathcal{C}$. We want to show there exists a function $u$ such that $e^{2u}g_\beta$ has Gaussian curvature $K$, where $g_\beta$ is the unique conformal conical metric with Gaussian curvature $1$. The existence of such a metric is equivalent to the existence of a solution $u$ to the equation 
 \begin{equation}\label{gaussian}
K = e^{-2u}(1 - \Delta_\beta u)
 \end{equation}

In the language of this section, it is enough to show that the restriction $\pi_0$ of the curvature map  to $\pi^{-1}(\mathcal{C})$ has $deg =1$. The assumption that $K \in \mathcal{C}$ guarantees that the map $\pi_0$ is a proper Fredholm map of index $0$ (see Theorems \ref{properness}, \ref{fredholmmap}). Observe that for given $\gamma, \alpha, k$ satisfying the conditions of the theorem, the subset of $\cdom$ consisting of positive functions is convex, thus  $\mathcal{C}$ is path-connected and there is a well-defined notion of degree. Clearly the function $K=1 \in \mathcal{C}$. On the other hand, the preimage of $K=1$ under $\pi_0$ is given by all solutions to the equation 
\begin{equation}\label{gaussian1}
1 = e^{-2u}(1 - \Delta_\beta u)
 \end{equation}
By Theorem 2 in \cite{luotian} there exists a unique conical metric $g$ on $S^2$ representing the divisor $\beta$ of constant curvature $1$. Since $u=0$ is a solution, it follows that the preimage $\pi_0^{-1}(1) = \{0\}$.

Now, the kernel of the differential of the curvature map $\pi_0$ under the assumption that $g_\beta$ has Gaussian curvature $1$ is given by solutions of 
\begin{equation}\label{kernel}
 D_0 \pi (h) = -2h - \Delta_\beta h = 0
\end{equation}
We now argue that the first eigenvalue $\lambda$ of the problem 
\begin{equation}
\Delta_\beta h = -\lambda h
\end{equation}
satisfies $\lambda \geq 2$. Moreover, if the lowest possible eigenvalue is achieved, namely $\lambda=2$,  then there exists a non-constant solution to the equation 
$Hess(f) = -f g$. 
 To see this, we follow the same ideas as in the works of Lichnerowicz and Obata \cite{lich, obata} which have now become standard. Using Bochner's formula away from the cone points, we can write 
 \begin{equation}
 \frac{1}{2}\Delta_\beta |\nabla h|^2 = |Hess (h)|^2 + g_\beta(\nabla \Delta h, \nabla h) + K_\beta |\nabla h|^2
 \end{equation}
Using Scharwz inequality and the fact that $h$ is an eigenfunction we get 
\begin{equation}
|Hess (h)|^2 \geq \frac{1}{2}( \Delta h)^2  = -\dfrac{\lambda}{2} h \Delta_\beta h
\end{equation}
Combining this with Bochner's formula we get the inequality 
\begin{align}\label{bochner}
\Delta_\beta |\nabla h|^2 \geq -\frac{\lambda}{2} h \Delta_\beta h  -\lambda |\nabla h|^2 + |\nabla h|^2
\end{align}
We now claim 
\begin{equation}\label{vanishlaplace}
\int_{S^2- \{p_1, \dots p_n\}} \Delta_\beta f dvol_\beta =0
\end{equation}
holds for any function $f \in \wlp{k-1}{\alpha}{\gamma-1}$. For $R>0$ small enough, let $B_R(p_k)$ be a geodesic ball of radius $R$ centered at the cone point $p_k$. Then 
\begin{align*}
\int_{S^2- \{p_1, \dots p_n\}} \Delta_\beta f dvol_\beta =  \int_{S^2- \{p_1, \dots p_n\}} \rho^{-2\beta}\Delta_{+1} f \rho^{2\beta}dvol_{+1} &= \lim_{R\to 0} \int_{S^2 - \cup_{k=1}^n B_R(p_k)} \Delta_{+1} f dvol_{+1}  \\
&=\lim_{R\to 0} \int_{S^2 - \cup_{k=1}^n B_R(p_k)} div(\nabla f) dvol_{+1}\\
&= \lim_{R\to 0} \sum_{k=1}^n \int_{S_R(p_k)} (\nabla f \cdot \nu ) dS\\
\end{align*}
where in the last equality we have used the divergence theorem, with  $S_R$ denoting the boundary of $B_R(p_k)$, $\nu$ the normal to each boundary circle and $dS$ the area element. As before, the assumption that $f \in \wlp{k-1}{\alpha}{\gamma-1}$ implies that $\sup_{S^2- \{p_1, \dots p_n\}} \rho^{-(\gamma-1) +1}|\nabla f| \leq C$ , so that 
\begin{align*}
\int_{S_R} |\rho^{-\gamma +2}\nabla f \cdot \nu| \rho^{\gamma-2}dS &\leq C' R^\gamma
\end{align*}
Since $\gamma>0$,  taking the limit as $R \to 0$ we get the desired result in (\ref{vanishlaplace}).

Now, using (\ref{vanishlaplace}) in combination with (\ref{bochner}), we obtain the inequality 
\begin{align*}
0 = \int_{S^2- \{p_1, \dots p_n\}} \Delta_\beta |\nabla h|^2 dvol_\beta &\geq \int_{S^2 - \{p_1, \dots p_n\}} -\frac{\lambda}{2} h \Delta_\beta h - \lambda |\nabla h|^2 + |\nabla h|^2 dvol_\beta\\
& =\l\frac{\lambda}{2} -\lambda +1\r \l\int_{S^2 -\{p_1, \dots p_n\}}  |\nabla h|^2 dvol_{\beta} \r 
\end{align*}
where the integration by parts is justified as in the proof of Theorem \ref{fredholmmap}. The previous inequality then shows that 
\begin{equation}
-\frac{\lambda}{2} +1 \leq 0
\end{equation}
so that $\lambda \geq 2$. Moreover, if $\lambda =2$, then the inequalities become equality, which forces the trace free part of the Hessian of $h$ to vanish, implying that $h$ solves the equation 
\begin{equation}\label{hessianeq}
Hess(h) = \phi g
\end{equation}
One can further show that $\phi = -h$.  Observe that (\ref{hessianeq}) implies the existence of a nonconstant solution to the equation 
\begin{equation}\label{conformalkilling}
\mathcal{L}_{\nabla h} g = -h g 
\end{equation}
Now suppose $h \in Ker (D_0\pi_0)$, that is, $h$ solves $-2h =\Delta_\beta h = \rho^{-2\beta}\Delta_{+1}h$, thus $h$ is an eigenfuction corresponding to the lowest possible eigenvalue. The discussion above shows that $h$ satisfies (\ref{conformalkilling}), in other words, $\nabla h$ is a conformal Killing field on $(S^2, g, \beta)$. Equivalently, this means the locally defined flow of $\nabla h$ preserves the conformal structure. Therefore, there exists a nontrivial one-parameter group of conformal transformations. Since the conformal group $Conf(S^2, \beta)$ is trivial, we must have $h =0$. Thus $K=1$ is a regular value of the curvature map $\pi_0$ . It now follows that $deg \pi_0 = 1$, as wanted. 
\end{proof}


\end{document}